\documentclass{amsart}

\usepackage{mathrsfs}
\usepackage{amscd}
\usepackage{amsmath}
\usepackage{amssymb}
\usepackage{amsthm}
\usepackage{bm}
\usepackage{epsf}
\usepackage{latexsym}
\usepackage{verbatim}
\usepackage[all, cmtip]{xy}
\usepackage{tikz}
\usetikzlibrary{positioning}
\usetikzlibrary{matrix}
\usepackage{float}
\usepackage{hyperref}
\usepackage{comment}
\usepackage{ytableau}

\usepackage[left=3.2cm, right=3.2cm]{geometry}

\newtheorem{theorem}{Theorem}[section]
\newtheorem{definition}[theorem]{Definition}
\newtheorem{lemma}[theorem]{Lemma}
\newtheorem{conjecture}[theorem]{Conjecture}
\newtheorem{corollary}[theorem]{Corollary}
\newtheorem{proposition}[theorem]{Proposition}

\newtheorem{varexample}[theorem]{Example}

\theoremstyle{definition}
\newtheorem{remark}[theorem]{Remark}

\newcommand{\PP}{\mathbb{P}}

\newcommand{\cD}{\mathcal{D}}

\newcommand{\cO}{\mathcal{O}}

\newcommand{\bmu}{\bm{\mu}}
\newcommand{\blam}{\bm{\lambda}}

\newcommand{\Trop}{\operatorname{Trop}}
\newcommand{\trop}{\operatorname{trop}}

\newcommand{\Pic}{\operatorname{Pic}}

\newcommand{\rk}{\mathrm{rk}}

\newcommand{\an}{\mathrm{an}}

\newenvironment{example}{\begin{varexample}
\begin{normalfont}}{\end{normalfont}
\end{varexample}}

\begin{document}
\title[Components of Brill-Noether Loci]{Components of Brill-Noether Loci for Curves with Fixed Gonality}
\author{Kaelin Cook-Powell}
\author{David Jensen}

\maketitle

\begin{abstract}
We describe a conjectural stratification of the Brill-Noether variety for general curves of fixed genus and gonality.  As evidence for this conjecture, we show that this Brill-Noether variety has at least as many irreducible components as predicted by the conjecture, and that each of these components has the expected dimension.  Our proof uses combinatorial and tropical techniques.  Specifically, we analyze containment relations between the various strata of tropical Brill-Noether loci identified by Pflueger in his classification of special divisors on chains of loops.
\end{abstract}

\section{Introduction}

Given a curve $C$ over the complex numbers, the \emph{Brill-Noether variety} $W^r_d (C)$ parameterizes line bundles of degree $d$ and rank at least $r$ on $C$.  Brill-Noether varieties encode a significant amount of geometric information, and consequently are among the most well-studied objects in the theory of algebraic curves.  A series of results in the eighties concern the geometry of $W^r_d(C)$ when $C$ is general in the moduli space $\mathcal{M}_g$.  In this case, the locally closed stratum $W^r_d(C) \smallsetminus W^{r+1}_d(C)$ is smooth \cite{Gieseker} of dimension
\[
\rho (g,r,d) := g-(r+1)(g-d+r) \hspace{.25in} \mbox{\cite{GH80}},
\]
and irreducible when $\rho(g,r,d)$ is positive \cite{FulLaz}.

More recent work has focused on the situation where $C$ is general in the Hurwitz space $\mathcal{H}_{k,g}$ parameterizing branched covers of the projective line of degree $k$ and genus $g$.  The Hurwitz space $\mathcal{H}_{k,g}$ admits a natural map to the moduli space $\mathcal{M}_g$, given by forgetting the data of the map to $\PP^1$.  When $k \geq \lfloor \frac{g+3}{2} \rfloor$, this map is dominant and there is nothing new to show, so we restrict our attention to the case where $k$ is smaller than $\lfloor \frac{g+3}{2} \rfloor$.  We refer to a general point in the Hurwitz space $\mathcal{H}_{k,g}$ as a \emph{general curve of genus $g$ and gonality $k$}.  Our main result is the following.

\begin{theorem}
\label{thm:Existence}
Let $C$ be a general curve of genus $g$ and gonality $k\geq 2$. Then there exists an irreducible component of $W^r_d(C)$ of dimension 
\[
\rho(g,\alpha-1,d)-(r+1-\alpha)k,
\]
as long as this number is nonnegative, for every positive integer $\alpha \leq \min \{ r+1, k-1 \}$ satisfying either $\alpha \geq k-(g-d+r)$ or $\alpha = r+1$.
\end{theorem}

We strongly suspect that Theorem~\ref{thm:Existence} identifies \emph{all} of the irreducible components of $W^r_d(C)$, for a reason that we will explain in Section~\ref{Sec:SplitStrat}.  Theorem~\ref{thm:Existence} is a generalization of several previous results.  In \cite{PfluegerkGonal}, Pflueger shows that the dimension of $W^r_d(C)$ is at most
\[
\rho_k (g,r,d) := \max_{\alpha} \rho(g,\alpha-1,d)-(r+1-\alpha)k,
\]
and asks whether every component has dimension $\rho(g,\alpha-1,d)-(r+1-\alpha)k$ for some value of $\alpha$.  In \cite{JensenRanganthan}, Ranganathan and the second author show that the maximal dimensional component has dimension exactly $\rho_k (g,r,d)$.  In \cite{CoppensMartens}, Coppens and Martens exhibit components of dimension $\rho(g,\alpha-1,d)-(r+1-\alpha)k$ for $\alpha$ equal to $1, r$, and $r+1$.  They further expand on this result in \cite{CoppensMartens2}, constructing components of dimension $\rho(g,\alpha-1,d)-(r+1-\alpha)k$ for all $\alpha$ dividing $r$ or $r+1$.

\subsection{The Splitting Type Stratification}
\label{Sec:SplitStrat}

Let $\pi : C \to \PP^1$ be a branched cover of degree $k$.  Given a line bundle $L$ on $C$, its pushforward $\pi_* L$ is a vector bundle of rank $k$ on $\PP^1$.  Every vector bundle on $\PP^1$ splits as a direct sum of line bundles
\[
\pi_* L \cong \cO (\mu_1) \oplus \cdots \oplus \cO(\mu_k)
\]
for some integers $\mu_1, \ldots , \mu_k$ that are unique up to permutation.

The vector $\bmu = (\mu_1 , \ldots , \mu_k)$ is known as the \emph{splitting type} of the vector bundle, and we write $\pi_\ast L\cong \cO(\bmu)$ for ease of notation.  We write $W^{\bmu}(C)$ for the locally closed subscheme parameterizing line bundles on $C$ whose pushforward has splitting type $\bmu$:
\[
W^{\bmu}(C) := \{ L \in \Pic (C) \vert \pi_* L \cong \cO (\bmu) \}.
\]

The splitting type of $\pi_*L$ determines not only the degree and rank of the line bundle $L$, but also the rank of $L \otimes \pi^* \cO(m)$ for all integers $m$ (see Section~\ref{Sec:Splitting}).  In this way, the varieties $W^{\bmu}(C)$ stratify $W^r_d(C)$.  The number of irreducible components of $W^r_d(C)$, as well as the dimensions of these components, are predicted by the following conjecture.  We refer the reader to Definition~\ref{Def:Order} for the definition of the partial order on splitting types, and to Definition~\ref{Def:Magnitude} for the definition of the magnitude of a splitting type.

\begin{conjecture}
\label{Conj:Strata}
Let $C$ be a general curve of genus $g$ and gonality $k \geq 2$.  Then:
\begin{enumerate}
    \item  $W^{\bmu}(C)$ is contained in the closure of $W^{\blam}(C)$ if and only if $\bmu \leq \blam$.
    \item  $W^{\bmu}(C)$ is smooth.
    \item  $W^{\bmu}(C)$ has dimension $g-\vert \bmu \vert$ if $g \geq \vert \bmu \vert$, and is empty otherwise.
    \item  $W^{\bmu}(C)$ is irreducible if $g > \vert \bmu \vert$.
\end{enumerate}
\end{conjecture}

At the time of writing, we learned of a simultaneous and independent proof of parts (1)-(3) of Conjecture~\ref{Conj:Strata}, due to H. Larson \cite{Larson}.

As evidence for Conjecture~\ref{Conj:Strata}, we consider the strata $W^{\bmu}(C)$ that the conjecture predicts to be maximal in $W^r_d (C)$.  For a given rank $r$ and degree $d$, the maximal elements of the poset of splitting types are in correspondence with positive integers $\alpha \leq \min \{ r+1, k-1 \}$ satisfying either $\alpha \geq k-(g-d+r)$ or $\alpha = r+1$.  (See Definition~\ref{Def:MaxSplit} and Proposition~\ref{Prop:Incomparable} for details.)  Let $\bmu_{\alpha}$ denote the splitting type corresponding to the integer $\alpha$.  Conjecture~\ref{Conj:Strata} predicts that the irreducible components of $W^r_d(C)$ are precisely the closures of the strata $W^{\bmu_{\alpha}}(C)$.  We prove the following stronger version of Theorem~\ref{thm:Existence}.

\begin{theorem}
\label{thm:Strata}
Let $C$ be a general curve of genus $g$ and gonality $k \geq 2$.  If $g \geq \vert \bmu_{\alpha} \vert$, then $W^{\bmu_{\alpha}}(C)$ has an irreducible component of dimension $g-\vert \bmu_{\alpha} \vert$.  The closure of this component is an irreducible component of $W^r_d (C)$.
\end{theorem}

\subsection{Approach and Techniques}

Our approach is based on tropical techniques developed in \cite{CDPR, PfluegerkGonal, PfluegerCycles, JensenRanganthan}.  Each of these papers establishes results about Brill-Noether varieties by studying the divisor theory of a particular family of metric graphs, known as the chains of loops. The first of these papers \cite{CDPR} provides a new proof of the Brill-Noether Theorem.  Key to this argument is the classification of special divisors on chains of loops $\Gamma$ with generic edge lengths.  Specifically, \cite{CDPR} shows that $W^r_d(\Gamma)$ is a union of tori $\mathbb{T}(t)$, where the tori are indexed by standard Young tableaux $t$.

In \cite{PfluegerCycles}, Pflueger generalizes this result to chains of loops with arbitrary edge lengths.  In this case, $W^r_d(\Gamma)$ is still a union of tori, but here the tori are indexed by a more general type of tableaux, known as \emph{displacement tableaux}.  (See Definition~\ref{Def:Displacement} and Theorem~\ref{thm:Classification}.)  In \cite{PfluegerkGonal}, Pflueger computes the dimension of the largest of these tori, and thus obtains his bound on the dimensions of Brill-Noether loci for general $k$-gonal curves.

Instead of studying the tori of maximum dimension, in this paper we study the tori that are maximal with respect to containment.  The tableaux corresponding to maximal-dimensional tori belong to a larger family, known as \emph{scrollar tableaux}.  (See Definition~\ref{Def:Scrollar}.)  There is a natural partition of scrollar tableaux into types, where the types are indexed by positive integers $\alpha \leq \min \{ r+1, k-1 \}$ satisfying either $\alpha \geq k-(g-d+r)$ or $\alpha = r+1$.  It is shown in \cite{JensenRanganthan} that, under certain mild hypotheses, divisor classes corresponding to scrollar tableaux lift to divisor classes on $k$-gonal curves in families of the expected dimension.

Our main combinatorial result is the following.

\begin{theorem}
\label{thm:MaxScrollar}
Let $\Gamma$ be a $k$-gonal chain of loops of genus $g$, and let $t$ be a $k$-uniform displacement tableau on $[r+1]\times[g-d+r]$.  The torus $\mathbb{T}(t)$ is maximal with respect to containment in $W^r_d(\Gamma)$ if and only if $t$ is scrollar.  In other words,
\[
W^r_d (\Gamma) = \bigcup_{t \text{ scrollar}} \mathbb{T}(t) .
\]
\end{theorem}

\subsection{Outline of the Paper}

Sections~\ref{Sec:Splitting} and~\ref{Sec:Divisors} contain preliminary material.  In Section~\ref{Sec:Splitting}, we review the basic theory of splitting types, and identify those that are maximal with respect to the dominance order.  In Section~\ref{Sec:Divisors}, we review the classification of special divisor classes on chains of loops from \cite{PfluegerkGonal, PfluegerCycles}, and the necessary results on scrollar tableaux from \cite{JensenRanganthan}.  In Section~\ref{Sec:Connections} we discuss the relation between our combinatorial and geometric results, and in particular show that Theorem~\ref{thm:MaxScrollar} implies Theorem~\ref{thm:Strata}.  In the final two sections, which are purely combinatorial, we prove Theorem~\ref{thm:MaxScrollar}.  In Section~\ref{Sec:Max}, we show that if $t$ is a scrollar tableau, then $\mathbb{T}(t)$ is maximal, and in Section~\ref{Sec:NonMax}, we establish the converse.

\subsection*{Acknowledgements}
We would like to thank Nathan Pflueger for several productive conversations, and for helping to formulate Conjecture~\ref{Conj:Strata}.  We also wish to thank Dhruv Ranganathan for comments on an early draft of this paper, and both Sam Payne and Ravi Vakil for helpful advice regarding the submission process.  The second author was supported by NSF DMS-1601896.

\section{Splitting Types}
\label{Sec:Splitting}

\subsection{Preliminary Definitions}
\label{Sec:SplitDefs}

In this section, we review the definition of splitting types and discuss some of their basic properties.  Let $\pi : C \to \PP^1$ be a branched cover of degree $k$ and genus $g$, and let $L$ be a line bundle on $C$.  As explained in Section~\ref{Sec:SplitStrat}, the pushforward $\pi_* L$ is a vector bundle of rank $k$ on $\PP^1$, and every vector bundle on $\PP^1$ splits as a direct sum of line bundles
\[
\pi_* L \cong \cO (\mu_1) \oplus \cdots \oplus \cO(\mu_k).
\]
The integers $\mu_1, \ldots , \mu_k$ are unique up to permutation.  We will assume throughout that
\[
\mu_1 \leq \mu_2 \leq \cdots \leq \mu_k.
\]
The vector $\bmu = (\mu_1 , \ldots , \mu_k)$ is known as the \emph{splitting type} of the vector bundle, and we write $\pi_\ast(L)\cong\cO(\bmu)$ for ease of notation.  It is helpful to think of a splitting type $\bmu$ as a partition with possibly negative parts.  This is because, for any $\ell$, the sum of the $\ell$ smallest entries of $\bmu$ is a lower semicontinuous invariant.  It is therefore natural to endow the set of splitting types with a partial order, extending the dominance order on partitions.

\begin{definition}
\label{Def:Order}
We define the \emph{dominance order} on splitting types as follows.  Let $\bmu$ and $\blam$ be splitting types satisfying $\sum_{i=1}^k \mu_i = \sum_{i=1}^k \lambda_i$.  We say that $\bmu \leq \blam$ if and only if
\[
\mu_1 + \cdots + \mu_{\ell} \leq \lambda_1 + \cdots + \lambda_{\ell} \hspace{.25 in} \mbox{ for all } \ell \leq k.
\]
\end{definition}

The splitting type of $\pi_* L$ determines the rank and degree of the line bundle $L$, as well as the rank of all its twists by line bundles pulled back from the $\PP^1$.  This can be seen by the Projection Formula, as follows:
\begin{align}
\label{eq:ranksequence}
\tag{$\star$} h^0(C,L \otimes \pi^*\cO_{\PP^1}(m)) &= h^0(\PP^1,\pi_* L \otimes \cO_{\PP^1}(m)) \\
&= \sum_{i=1}^{k} h^0(\PP^1,\cO_{\PP^1}(\mu_i + m)) \notag \\
&= \sum_{i=1}^{k} \max \{ 0, \mu_i +m+1 \} . \notag
\end{align}
In particular, we have
\begin{align}
h^0(L) &= \sum_{i=1}^k \max \{ 0, \mu_i +1 \} \hspace{.2 in} \mbox{ and }\notag \\
\deg L &= g+k-1+ \sum_{i=1}^k \mu_i. \notag
\end{align}
This suggests the following definition.

\begin{definition}
Let $W^{\bmu}(C)$ denote the locally closed subscheme parameterizing line bundles on $C$ whose pushforward has splitting type $\bmu$:
\[
W^{\bmu}(C) := \{ L \in \Pic (C) \vert \pi_* L \cong \cO (\bmu) \}.
\]
\end{definition}

\noindent The expected codimension of $W^{\bmu}(C)$ in $\Pic^d (C)$ is given by the \emph{magnitude} of $\bmu$.

\begin{definition}
\label{Def:Magnitude}
The \emph{magnitude} of a splitting type $\bmu$ is
\[
\vert \bmu \vert := \sum_{i<j} \max \{ 0, \mu_j - \mu_i -1 \}.
\]
\end{definition}

\begin{example}
Let $C$ be a trigonal curve of genus 5.  We will show that $W^1_4 (C)$ has 2 irreducible components, both isomorphic to $C$.  First, there is a 1-dimensional family of rank 1 divisor classes obtained by adding a basepoint to the $g^1_3$.  If $D \in W^1_4 (C)$ is not in this 1-dimensional family, then $D-g^1_3$ is not effective.  It follows from the basepoint free pencil trick that the multiplication map
\[
\nu : H^0 (D) \otimes H^0 (g^1_3) \to H^0 (D+g^1_3)
\]
is injective.  The divisor class $D+g^1_3$ is therefore special.  From this we see that the Serre dual $K_C - D$ is a divisor class in $W^1_4 (C)$ with the property that $(K_C - D) - g^1_3$ is effective.

We therefore see that $W^1_4 (C)$ has two components, both isomorphic to $C$, as pictured in Figure~\ref{Fig:Trigonal}.  One of these components consists of divisor classes $D$ such that $D-g^1_3$ is effective, and the other component consists of the Serre duals of classes in the first component.  Since $K_C - 2g^1_3$ is effective of degree 2, we see that these two components intersect in 2 points.

Alternatively, this analysis can be carried out by examining the splitting type stratification of $W^1_4 (C)$.  By (\ref{eq:ranksequence}), we see that line bundles in the first component, in the complement of the two intersection points, have splitting type $(-2,-2,1)$.  Similarly, line bundles in the second component, in the complement of the two intersection points, have splitting type $(-3,0,0)$.  Finally, the two line bundles in the intersection have splitting type $(-3,-1,1)$.  Notice that this third splitting type is smaller than each of the previous two in the dominance order, and that the codimension of each stratum in $\Pic^4 (C)$ is the magnitude of the splitting type.

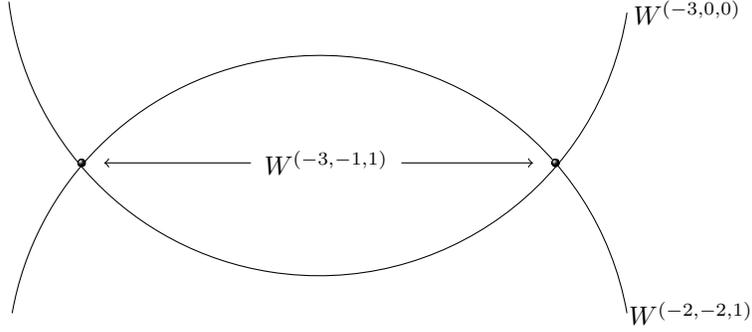
\begin{figure}[h!]
\begin{tikzpicture}
\draw (0,0) arc[radius = 4.15, start angle=10, end angle=170];
\draw (0,4) arc[radius = 4.15, start angle=-9, end angle=-173];

\draw (0.8,4) node {$W^{(-3,0,0)}$};
\draw (0.85,0) node {$W^{(-2,-2,1)}$};
\draw (-4,2) node {$W^{(-3,-1,1)}$};

\draw [ball color=black] (-0.95,2) circle (0.5mm);
\draw [ball color=black] (-7.25,2) circle (0.5mm);

\draw[->] (-3,2)--(-1.25,2);
\draw[->] (-5,2)--(-6.95,2);

\end{tikzpicture}
\caption{Stratification of $W^1_4$ for a general curve of genus 5 and gonality 3.}
\label{Fig:Trigonal}
\end{figure}

\end{example}

\subsection{Maximal Splitting Types}
\label{Sec:MaxSplit}

For the remainder of this section, we fix positive integers $g$, $r$, $d$, and $k$ such that $r > d-g$.  Among the possible splitting types of line bundles of degree $d$ and rank at least $r$ on a $k$-gonal curve of genus $g$, we identify those that are maximal with respect to the dominance order.

\begin{definition}
\label{Def:MaxSplit}
Let $\alpha \leq \min \{ r+1, k-1 \}$ be a positive integer.  By the division algorithm, there exists a unique pair of integers $q,\beta$ such that
\[
r+1 = q\alpha + \beta, \hspace{.2 in} 0 \leq \beta < \alpha .
\]
Similarly, there exists a unique pair of integers $q',\beta'$ such that
\[
g-d+r = q'(k-\alpha) + \beta', \hspace{.2 in} 0 \leq \beta' < k-\alpha .
\]
We define the splitting type $\bmu_{\alpha}$ as follows:
\begin{displaymath}
\mu_{\alpha,i} := \left\{ \begin{array}{ll}
-q'-2 & \textrm{if $0 < i \leq \beta'$} \\
-q'-1 & \textrm{if $\beta' < i \leq k-\alpha$} \\
q-1 & \textrm{if $k-\alpha < i \leq k-\beta$} \\
q & \textrm{if $k-\beta < i \leq k$.}
\end{array}\right.
\end{displaymath}

\end{definition}

Heuristically, $\bmu_{\alpha}$ is the ``most balanced'' splitting type of degree $d$ and rank $r$, subject to the constraint that precisely $\alpha$ of its entries are nonnegative.  We show that the expected codimension of $W^{\bmu_{\alpha}}(C)$ coincides with the dimensions of irreducible components of $W^r_d (C)$ predicted by \cite[Question~1.12]{PfluegerkGonal}.

\begin{lemma}
\label{Lem:ExpectDim}
For any integer $\alpha$, we have
\[
g - \vert \bmu_{\alpha} \vert = \rho(g,\alpha-1,d)-(r+1-\alpha)k.
\]
\end{lemma}

\begin{proof}
First, recall that
\[
\vert \bmu_{\alpha} \vert = \sum_{i<j} \max \{ 0, \mu_{\alpha,j} - \mu_{\alpha,i} -1 \}.
\]
If $i<j \leq k-\alpha$, then $\mu_{\alpha,j} - \mu_{\alpha,i} \leq 1$, so the pair $(i,j)$ does not contribute to the sum above.  Similarly, if $k-\alpha < i < j$, then $\mu_{\alpha,j} - \mu_{\alpha,i} \leq 1$, so again the pair $(i,j)$ does not contribute to the sum above.

On the other hand, if $i \leq k-\alpha$ and $j > k-\alpha$, then the pair $(i,j)$ does contribute to the sum.  There are precisely $(k-\alpha)\alpha$ such pairs, each $\mu_{\alpha,i}$ with $i \leq k-\alpha$ appears in exactly $\alpha$ of these pairs, and each $\mu_{\alpha,j}$ with $j > k-\alpha$ appears in exactly $k-\alpha$ of these pairs.  It follows that we may rewrite the sum above as
\begin{align*}
\vert \bmu \vert &= (k-\alpha) \sum_{j=k-\alpha+1}^k \mu_j -\alpha \sum_{i=1}^{k-\alpha} \mu_i - (k-\alpha)\alpha \notag \\
&= (k-\alpha)(r+1-\alpha) + \alpha(g-d+r+k-\alpha) - (k-\alpha)\alpha \notag \\
&= \alpha(g-d+\alpha-1)+(r+1-\alpha)k. \notag
\end{align*}
Subtracting both sides from $g$ yields the result.
\end{proof}

Recall that the integers $g,r,d,$ and $k$ are fixed.  We will say that a splitting type is \emph{maximal} if it is maximal with respect to the dominance order among all splitting types satisfying
\[
\sum_{i=1}^k \mu_i = d+1-g-k
\]
and
\[
\sum_{i=1}^k \max \{ 0, \mu_i + 1 \} \geq r+1.
\]
In the rest of this section, we show that the maximal splitting types are precisely the splitting types $\bmu_{\alpha}$, when either $\alpha \geq k-(g-d+r)$ or $\alpha = r+1$.  We first prove the following reduction step.

\begin{lemma}
\label{Lem:Reduction}
A maximal splitting type $\bmu$ satisfies
\[
\sum_{i=1}^k \max \{ 0, \mu_i + 1 \} = r+1.
\]
\end{lemma}

\begin{proof}
For the purposes of this argument, we define
\[
h(\bmu) = \sum_{i=1}^k \max \{ 0, \mu_i + 1 \}.
\]
Let $\bmu$ be a splitting type satisfying
\[
\sum_{i=1}^k \mu_i = d+1-g-k
\]
and $h(\bmu) \geq r+1$.  We will show, by induction on $h(\bmu)$, that there exists a splitting type $\blam$ such that $\bmu \leq \blam$ and $h(\blam) = r+1$. 

Since $r \geq 0$, we see that $\mu_k \geq 0$, and since $h(\bmu) \geq r+1 > d-g+1$, we see that $\mu_1 < -1$.  There therefore exists an integer $i$ such that $\mu_i > \mu_{i-1}$.  Let $j$ be the smallest such integer and $j'$ the largest such integer.  Since $\mu_1 < -1$ and $\mu_k \geq 0$, either $j < j'$, or $j=j'$ and $\mu_{j-1} < \mu_j -1$.  It follows that the vector $\bmu'$ obtained from $\bmu$ by adding 1 to $\mu_{j-1}$ and subtracting 1 from $\mu_{j'}$ is nondecreasing, and therefore a valid splitting type.  Moreover, we have $\bmu < \bmu'$.  Since $\mu_{j-1} < -1$ and $\mu_{j'} \geq 0$, we see that $h(\bmu') = h(\bmu) - 1$, and the result follows by induction.
\end{proof}

We now show that every maximal splitting type is of the form $\bmu_{\alpha}$ for some $\alpha$.

\begin{lemma}
\label{Lem:MaxSplit}
Let $\bmu$ be a splitting type satisfying
\[
\sum_{i=1}^k \mu_i = d+1-g-k
\]
and
\[
\sum_{i=1}^k \max \{ 0, \mu_i + 1 \} = r+1.
\]
Let $\alpha$ denote the number of nonnegative entries of $\bmu$.  Then $\bmu \leq \bmu_{\alpha}$.
\end{lemma}

\begin{proof}
By assumption, we have
\[
\sum_{i=k-\alpha+1}^k \mu_i = r+1-\alpha = \sum_{i=k-\alpha+1}^k \mu_{\alpha,i}.
\]
It follows that
\[
\sum_{i=1}^{k-\alpha} \mu_i = -(g-d+r)-(k-\alpha) = \sum_{i=1}^{k-\alpha} \mu_{\alpha,i}.
\]
Because the entries of $\bmu$ are ordered from smallest to largest, for any $\ell \leq k-\alpha$, we see that
\[
\sum_{i=1}^{\ell} \mu_i \leq \frac{\ell}{k-\alpha} \sum_{i=1}^{k-\alpha} \mu_i = \frac{-\ell(g-d+r)}{k-\alpha}-\ell. 
\]
Similarly, for any $\ell \leq \alpha$, we see that
\[
\sum_{i=k-\alpha+1}^{k-\alpha+\ell} \mu_i \leq \frac{\ell}{\alpha} \sum_{i=k-\alpha+1}^k \mu_i = \frac{\ell(r+1)}{\alpha}-\ell .
\]
By definition of $\bmu_{\alpha}$, therefore, we have $\bmu \leq \bmu_{\alpha}$.
\end{proof}

\begin{corollary}
\label{Cor:MaxSplit}
If $\bmu$ is a maximal splitting type, then $\bmu = \bmu_{\alpha}$ for some integer $\alpha$.
\end{corollary}

\begin{proof}
Let $\bmu$ be a maximal splitting type. By Lemma~\ref{Lem:Reduction}, we see that 
\[
\sum_{i=1}^k \max \{ 0, \mu_i + 1 \} = r+1.
\]
Let $\alpha$ denote the number of nonnegative entries of $\bmu$.  By Lemma~\ref{Lem:MaxSplit}, we have $\bmu \leq \bmu_{\alpha}$, but since $\bmu$ is maximal, it follows that $\bmu = \bmu_{\alpha}$.
\end{proof}

We now show that, if $\alpha < \min \{k-(g-d+r),r+1 \}$, then $\bmu_{\alpha}$ is not maximal.

\begin{lemma}
\label{Lem:LowerBoundOnAlpha}
If $\alpha < \min \{k-(g-d+r),r+1 \}$, then $\bmu_{\alpha} < \bmu_{\alpha+1}$.
\end{lemma}

\begin{proof}
Since $g-d+r < k-\alpha$, by definition we have $\mu_{\alpha,k-\alpha} = -1$.  If $r+1$ is not divisible by $\alpha$, consider the splitting type $\bmu$ obtained from $\bmu_{\alpha}$ by adding 1 to $\mu_{\alpha,k-\alpha}$ and subtracting 1 from $\mu_{\alpha,k-\beta+1}$.  On the other hand, if $r+1$ is divisible by $\alpha$, then since $\alpha < r+1$, we must have $\mu_{\alpha,k-\alpha+1} > 0$.  In this case, consider the splitting type $\bmu$ obtained from $\bmu_{\alpha}$ by adding 1 to $\mu_{\alpha,k-\alpha}$ and subtracting 1 from $\mu_{\alpha,k-\alpha+1}$.  In either case, we see that $\bmu$ is a splitting type with $\alpha+1$ nonnegative entries, satisfying $\bmu_{\alpha} < \bmu$.  By Lemma~\ref{Lem:MaxSplit}, we have $\bmu_{\alpha} < \bmu \leq \bmu_{\alpha+1}$.
\end{proof}

Finally, we see that the remaining splitting types $\bmu_{\alpha}$ are maximal.

\begin{proposition}
\label{Prop:Incomparable}
The splitting type $\bmu$ is maximal if and only if $\bmu = \bmu_{\alpha}$ for some integer $\alpha$ satisfying either $\alpha \geq k-(g-d+r)$ or $\alpha = r+1$.
\end{proposition}

\begin{proof}
By Corollary~\ref{Cor:MaxSplit}, every maximal splitting type is of the form $\bmu_{\alpha}$ for some integer $\alpha$.  By Lemma~\ref{Lem:LowerBoundOnAlpha}, if $\alpha < k-(g-d+r)$ and $\alpha \neq r+1$, then $\bmu_{\alpha}$ is not maximal.  It therefore suffices to show that, if $\alpha \neq \gamma$ are both greater than or equal to $k-(g-d+r)$, then $\bmu_{\alpha}$ and $\bmu_{\gamma}$ are incomparable.

Without loss of generality, assume that $\alpha < \gamma$.  We write
\begin{align*}
r+1 &= q_{\alpha} \alpha + \beta_{\alpha} & 0 \leq \beta_{\alpha} < \alpha \notag \\
&= q_{\gamma} \gamma + \beta_{\gamma} & 0 \leq \beta_{\gamma} < \gamma . \notag
\end{align*}
Since $\alpha < \gamma$, we see that $q_{\alpha} \geq q_{\gamma}$.  Moreover, since $\gamma \leq r+1$, we see that both $q_{\alpha}$ and $q_{\gamma}$ are positive.  It follows that, if $q_{\alpha} = q_{\gamma}$, then $\beta_{\alpha} > \beta_{\gamma}$.  Thus, if $j$ is the largest integer such that $\mu_{\alpha,j} \neq \mu_{\gamma,j}$, then $\mu_{\alpha,j} > \mu_{\gamma,j}$.  If $\bmu_{\alpha}$ and $\bmu_{\gamma}$ are comparable, then we see that $\bmu_{\alpha} < \bmu_{\gamma}$.

Since $k-\alpha \leq g-d+r$, we see by a similar argument that if $j'$ is the smallest integer such that $\mu_{\alpha,j'} \neq \mu_{\gamma,j'}$, then $\mu_{\alpha,j'} > \mu_{\gamma,j'}$.  It follows that if $\bmu_{\alpha}$ and $\bmu_{\gamma}$ are comparable, then $\bmu_{\alpha} > \bmu_{\gamma}$.  Combining these two observations, we see that $\bmu_{\alpha}$ and $\bmu_{\gamma}$ are incomparable.
\end{proof}

\section{Divisor Theory of Chains of Loops}
\label{Sec:Divisors}

In this section, we survey the theory of special divisors on chains of loops, as discussed in \cite{PfluegerkGonal, PfluegerCycles, JensenRanganthan}.  We refer the reader to those papers for more details.  For a more general overview of divisors on tropical curves, we refer the reader to \cite{Baker, BakerJensen}.  For the uninitiated, we will not require most of the material of these papers; we will use only the classification of special divisors on chains of loops from \cite{PfluegerkGonal, PfluegerCycles}.

\subsection{Chains of Loops and Torsion Profiles}

Let $\Gamma$ be a chain of $g$ loops with bridges, as pictured in Figure~\ref{Fig:TheGraph}.  Each of the $g$ loops consists of two edges.  We denote the lengths of the top and bottom edge of the $j$th loop by $\ell_j$ and $m_j$, respectively.  The Brill-Noether theory of chains of loops is governed by the \emph{torsion orders} of the loops.

\begin{figure}[h!]
\begin{tikzpicture}

\draw (0,0) circle (1);
\draw (1,0)--(2,0);
\draw (3,0) circle (1);
\draw (4,0)--(5,0);
\draw (6,0) circle (1);
\draw (7,0)--(8,0);
\draw (9,0) circle (1);
\draw (10,0)--(11,0);
\draw (12,0) circle (1);

\draw [<->] (7.15,0.5) arc[radius = 1.15, start angle=10, end angle=170];
\draw [<->] (7.15,-0.5) arc[radius = 1.15, start angle=-9, end angle=-173];

\draw (6,1.75) node {\footnotesize$\ell_j$};
\draw (6,-1.75) node {\footnotesize$m_j$};

\draw [ball color=black] (-1,0) circle (0.5mm);
\draw [ball color=black] (1,0) circle (0.5mm);

\draw [ball color=black] (2,0) circle (0.5mm);
\draw [ball color=black] (4,0) circle (0.5mm);

\draw [ball color=black] (5,0) circle (0.5mm);
\draw [ball color=black] (7,0) circle (0.5mm);

\draw [ball color=black] (8,0) circle (0.5mm);
\draw [ball color=black] (10,0) circle (0.5mm);

\draw [ball color=black] (11,0) circle (0.5mm);
\draw [ball color=black] (13,0) circle (0.5mm);
\end{tikzpicture}
\caption{The chain of loops $\Gamma$.}
\label{Fig:TheGraph}
\end{figure}
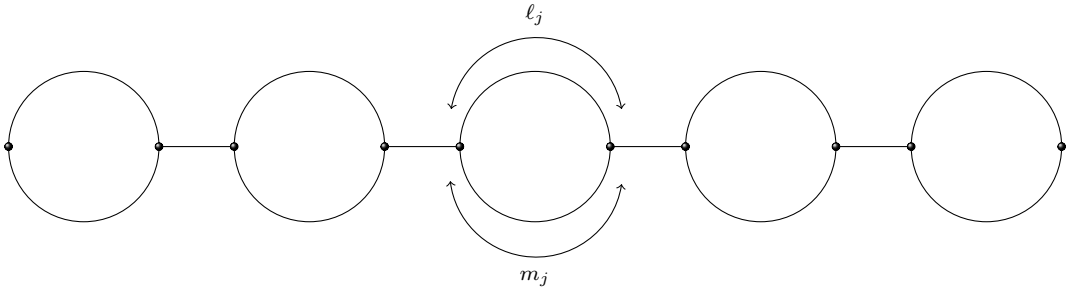

\begin{definition} \cite[Definition~1.9]{PfluegerCycles}
\label{Def:TorsionOrder}
If $\ell_j + m_j$ is an irrational multiple of $m_j$, then the $j$th \emph{torsion order} $\tau_j$ of $\Gamma$ is 0.  Otherwise, we define $\tau_j$ to be the minimum positive integer such that $\tau_j m_j$ is an integer multiple of $\ell_j + m_j$.  The sequence $\bm{\tau} = (\tau_1 , \ldots , \tau_g)$ is called the \emph{torsion profile} of $\Gamma$.
\end{definition}

\noindent For the remainder of this paper, we assume that the torsion profile of $\Gamma$ is given by
\begin{displaymath}
\tau_i := \left\{ \begin{array}{ll}
0 & \textrm{if $i<k$ or $i>g-k+1$} \\
k & \textrm{otherwise.}
\end{array}\right.
\end{displaymath}
This chain of loops with this torsion profile possesses a distinguished divisor class of rank 1 and degree $k$, given by $g^1_k = kv_k$, where $v_k$ is the lefthand vertex of the $k$th loop.

\begin{remark}
Note that, unlike \cite[Definition~2.1]{PfluegerkGonal}, we do \emph{not} require the first $k-1$ loops or the last $k-1$ loops to have torsion order $k$.  This choice does not affect the gonality, or more generally the Brill-Noether theory, of this metric graph.  A primary reason for this choice is that the space of such metric graphs has dimension equal to that of the Hurwitz space, namely $2g+2k-5$.
\end{remark}

In \cite{PfluegerCycles}, Pflueger classifies the special divisor classes on chains of loops.  This classification generalizes that of special divisor classes on generic chains of loops in \cite{CDPR}.  Specifically, Pflueger shows that $W^r_d (\Gamma)$ is a union of tori, where the tori are indexed by certain types of tableaux.  While Pflueger's analysis applies to chains of loops with arbitrary torsion profiles, we record it only for the torsion profile above.  For ease of notation, given a positive integer $a$ we write $[a]$ for the finite set $\{ 1, \ldots , a \}$.

\begin{definition} \cite[Definition~2.5]{PfluegerkGonal}
\label{Def:Displacement}
Let $a$ and $b$ be positive integers.  Recall that a \emph{tableau} on $[a]\times[b]$ with alphabet $[g]$ is a function $t: [a]\times[b] \to [g]$ satisfying:
\[
t(x,y) < t(x,y+1) \mbox{ and } t(x,y) < t(x+1,y) \mbox{ for all } (x,y).
\]
A tableau $t$ is \emph{standard} if $t$ is injective.  A tableau $t$ is called a \emph{$k$-uniform displacement tableau} if, whenever 
\[
t(x,y) = t(x',y'), \mbox{ we have } x-y = x'-y' \pmod{k}.
\]
\end{definition}

It is standard to depict a tableau on $[a]\times[b]$ as a rectangle with $a$ columns and $b$ rows, where the box in position $(x,y)$ is filled with the symbol $t(x,y)$.  We draw our tableaux according to the English convention, so that the box $(1,1)$ appears in the upper lefthand corner.

\subsection{Coordinates on $\Pic (\Gamma)$}

A nice feature of the chain of loops is that its Picard group has a natural system of coordinates.  On the $j$th loop, let $\langle \xi \rangle_j$ denote the point located $\xi m_j$ units from the righthand vertex in the counterclockwise direction.  Note that
\[
\langle \xi \rangle_j = \langle \eta \rangle_j \mbox{ if and only if } \xi = \eta \pmod{\tau_j}. 
\]
By the tropical Abel-Jacobi theorem \cite{BakerNorine}, every divisor class $D$ of degree $d$ on $\Gamma$ has a \emph{unique} representative of the form
\[
(d-g)\langle 0 \rangle_g + \sum_{j=1}^g \langle \xi_j (D) \rangle_j , 
\]
for some real numbers $\xi_j (D)$.  Because this expression is unique, the functions $\xi_j$ form a system of coordinates on $\Pic^d (\Gamma)$.  This representative of the divisor class $D$ is known as the \emph{break divisor} representative \cite{MikhalkinZharkov08,ABKS}.

\begin{definition} \cite[Definition~3.5]{PfluegerCycles}
Given a degree $d$ and a $k$-uniform displacement tableau $t$ with alphabet $[g]$, we define the coordinate subtorus $\mathbb{T}(t)$ as follows.
\[
\mathbb{T}(t) := \{ D \in \Pic^d (\Gamma) \vert \xi_{t(x,y)} (D) = y-x \pmod{k} \} .
\]
\end{definition}

Note that the coordinate $\xi_j (D)$ of a divisor class $D$ in $\mathbb{T}(t)$ is determined if and only if $j$ is in the image of $t$.  It follows that the codimension of $\mathbb{T}(t)$ in $\Pic^d (\Gamma)$ is the number of distinct symbols in $t$.  The main combinatorial result of \cite{PfluegerCycles} is a classification of special divisors on $\Gamma$.

\begin{theorem} \cite[Theorem~1.4]{PfluegerCycles}
\label{thm:Classification}
For any positive integers $r$ and $d$ satisfying $r > d-g$, we have
\[
W^r_d (\Gamma) = \bigcup \mathbb{T}(t),
\]
where the union is over $k$-uniform displacement tableaux on $[r+1]\times[g-d+r]$ with alphabet $[g]$.
\end{theorem}

Notably, Pflueger does not consider the containment relations between the various tori $\mathbb{T}(t)$.  These containment relations are the primary concern of Sections~\ref{Sec:Max} and~\ref{Sec:NonMax}.  We note the following, which will be explored in more detail in these later sections.

\begin{lemma}\label{lem:Containment}
Let $t$ and $t'$ be $k$-uniform displacement tableaux on $[a]\times[b]$.  Then $\mathbb{T}(t)\subseteq \mathbb{T}(t')$ if and only if
\begin{enumerate}
    \item every symbol in $t'$ is a symbol in $t$, and
    \item if $t(x,y)=t'(x',y')$, then $x-y = x'-y'\pmod k$.
\end{enumerate}
\end{lemma}

Under Pflueger's classification of special divisors, there is a natural interpretation of Serre duality.  Given a tableau $t$ on $[a]\times[b]$, define the \emph{transpose tableau} to be the tableau $t^T$ on $[b]\times[a]$ given by $t^T(x,y) = t(y,x)$.

\begin{lemma}\cite[Remark~3.6]{PfluegerCycles}
Let $t$ be a $k$-uniform displacement tableau on $[r+1] \times [g-d+r]$ with alphabet $[g]$, and let $D \in \mathbb{T}(t)$ be a divisor class.  Then the Serre dual $K_{\Gamma} - D$ is contained in $\mathbb{T}(t^T)$.
\end{lemma}

\subsection{Scrollar Tableaux}

In \cite{JensenRanganthan}, Ranganathan and the second author consider a special type of $k$-uniform displacement tableaux, known as \emph{scrollar tableaux}.  Throughout this section, we fix positive integers $a$ and $b$, and a positive integer $\alpha \leq \{ a,k-1 \}$, satisfying either $\alpha \geq k-b$ or $\alpha = a$.  As in Definition~\ref{Def:MaxSplit}, we write
\[
a = q\alpha + \beta, \hspace{.2 in} 0 \leq \beta < \alpha 
\]
and
\[
b = q'(k-\alpha) + \beta', \hspace{.2 in} 0 \leq \beta' < k-\alpha .
\]

\begin{definition}
\label{Def:Scrollar}
Let $t$ be a tableau on $[a]\times[b]$.  We define $t$ to be \emph{scrollar of type} $\alpha$ if it satisfies the following three conditions.
\begin{enumerate}
    \item $t(x,y) = t(x',y')$ if and only if there exists an integer $\ell$ such that both 
    \[
    x'-x = \ell\alpha \mbox{ and } y'-y = \ell(\alpha-k).
    \]
    
    \item If $\alpha = a$, then $t(1,y) > t(a,y+a-k)$ for all $y > k-a$.
    
    \item If $\alpha = k-b$, then $t(x,1) > t(x+b-k,b)$ for all $x > k-b$.
\end{enumerate}
\end{definition}

\begin{remark}
\label{Rmk:DiffDef}
When $k-b < \alpha < a$, Definition~\ref{Def:Scrollar} agrees with \cite[Definition~7.1]{JensenRanganthan}, but in the edge cases the two definitions disagree.  This is because, when $\alpha$ is equal to $a$ or $k-b$, every standard tableau satisfies \cite[Definition~7.1]{JensenRanganthan} trivially.  In Sections~\ref{Sec:Max} and~\ref{Sec:NonMax}, however, we will see that $\mathbb{T}(t)$ is maximal only for tableaux satisfying Definition~\ref{Def:Scrollar}.  We note that when $\alpha < a$, condition (1) implies an inequality analogous to that of condition (2), because
\[
t(1,y) = t(\alpha+1,y+\alpha-k) > t(\alpha,y+\alpha-k).
\]
Similarly, when $\alpha > k-b$, condition (1) implies an inequality analogous to that of condition (3).

For the reader interested in comparing the definitions in the two papers, we provide a brief dictionary.  The integer $\alpha$ appearing here is the same as $n$ in \cite{JensenRanganthan}.  The integer $\beta$ agrees with $b$ in \cite{JensenRanganthan}, and $q$ is equal to $\lfloor \frac{a}{\alpha} \rfloor = \lfloor \frac{r+1}{n} \rfloor$.
\end{remark}

\begin{example}
A typical example of a scrollar tableau appears in Figure~\ref{Fig:Scroll}.  Note that the boxes in the first $\alpha$ columns necessarily contain distinct symbols, as do the boxes in the last $k-\alpha$ rows.  The symbols in the remaining boxes are obtained by repeatedly translating the symbols in this L-shaped region $\alpha$ boxes rightward and $k-\alpha$ boxes upward.

\begin{figure}[H]
\begin{ytableau}
*(gray)1 & *(gray)2 & *(gray)4 & 5 & 10 & 11 & 12 \\
*(gray)3 & *(gray)7 & *(gray)8 & 9 & 13 & 16 & 18 \\
*(gray)5 & *(gray)10 & *(gray)11 & 12 & 15 & 17 & 20 \\
*(gray)9 & *(gray)13 & *(gray)16 & *(gray)18 & *(gray)19 & *(gray)22 & *(gray)23 \\
*(gray)12 & *(gray)15 & *(gray)17 & *(gray)20 & *(gray)21 & *(gray)24 & *(gray)26 \\
\end{ytableau}
\caption{A scrollar tableau of type 3, where $k=5$.}
\label{Fig:Scroll}
\end{figure}
\end{example}

\begin{example}
Figure~\ref{Fig:3Tab} depicts three different $3$-uniform displacement tableaux on $[3]\times[2]$.  The first tableau $t$ is scrollar of type 2.  To see this, note that there is only one pair of boxes whose $x$ coordinates differ by a multiple of 2 and whose $y$ coordinates differ by the same multiple of $-1$, and these boxes contain the same symbol.  The second tableau $t'$ is scrollar of type 1, because it is standard, $t'(2,1) >t'(1,2)$, and $t'(3,1) > t'(2,2)$.  The final tableau $t^\ast$ is not scrollar of either type.  Specifically, it is not scrollar of type 1 because $t^\ast (2,1) < t^\ast (1,2)$, and it is not scrollar of type 2 because $t^\ast (3,1) \neq t^\ast (1,2)$.  By Lemma~\ref{lem:Containment}, we see that $\mathbb{T}(t^\ast)\subset \mathbb{T}(t)$.

\begin{figure}[H]
$t=\text{ }$\begin{ytableau}
1 & 2 & 4\\
4 & 5 & 6
\end{ytableau}\hskip2ex $t'=\text{ }$\begin{ytableau}
1 & 3 & 5\\
2 & 4 & 6
\end{ytableau}\hskip2ex $t^\ast=\text{ }$\begin{ytableau}
1 & 2 & 3\\
4 & 5 & 6
\end{ytableau}
\caption{Three different 3-uniform displacement tableaux.  The first two are scrollar of different types, and the third is not scrollar.}
\label{Fig:3Tab}
\end{figure}
\end{example}

The following observation from \cite{JensenRanganthan} is central to our argument.

\begin{proposition}
\label{Prop:ScrollarDim}
Let $t$ be a scrollar tableau of type $\alpha$ on $[r+1]\times[g-d+r]$ with alphabet $[g]$.  Then $g \geq \vert \bmu_{\alpha} \vert$ and
\[
\dim \mathbb{T}(t) = g - \vert \bmu_{\alpha} \vert .
\]
\end{proposition}

\begin{proof}
By \cite[Proposition~7.4]{JensenRanganthan}, we have
\[ 
\dim \mathbb{T}(t) = \rho(g,\alpha-1,d)-(r+1-\alpha)k.
\]
The result then follows from Lemma~\ref{Lem:ExpectDim}.
\end{proof}

Proposition~\ref{Prop:ScrollarDim} suggests a connection between scrollar tableaux of type $\alpha$ and the splitting type $\bmu_{\alpha}$.  This connection will be established in Proposition~\ref{Prop:ScrollSpecialize} below.  The following lemma is key to the proof of Proposition~\ref{Prop:ScrollSpecialize}.

\begin{lemma} \cite[Corollary~7.3]{JensenRanganthan}
\label{Lem:JRBounds}
Let $t$ be a scrollar tableau of type $\alpha$, and let $D \in \mathbb{T}(t)$ be a sufficiently general divisor class.  Then
\begin{enumerate}
    \item $\rk (D-qg^1_k) = \beta - 1$, and
    \item $\rk (D-(q+1)g^1_k) = -1$.
\end{enumerate}
\end{lemma}

\begin{remark}
\label{Rmk:SuffGeneral}
In Lemma~\ref{Lem:JRBounds}, when we say that the divisor class $D \in \mathbb{T}(t)$ is ``sufficiently general'', we mean that $D$ lies in the complement of finitely many coordinate subtori of codimension at least 1 in $\mathbb{T}(t)$.  In particular, the set of divisor classes in $\mathbb{T}(t)$ satisfying the conclusion of Lemma~\ref{Lem:JRBounds} is open and dense in $\mathbb{T}(t)$.
\end{remark}

Much of \cite{JensenRanganthan} is devoted to a lifting result for divisor classes in $\mathbb{T}(t)$ when $t$ is a scrollar tableau.  Unfortunately, \cite{JensenRanganthan} does not establish this lifting result for all scrollar tableaux, but only for those that satisfy the following condition.

\begin{definition}
\label{Def:VerticalSteps}
We say that a tableau $t$ \emph{has no vertical steps} if
\[
t(x,y+1) \neq t(x,y)+1 \mbox{ for all } x,y \mbox{ such that } k \leq t(x,y) \leq g-k.
\]
\end{definition}

We note that if $g \geq \vert \bmu_{\alpha} \vert$ and $\alpha > 1$, then there exists a scrollar tableau of type $\alpha$ with no vertical steps.  For example, the transpose of the tableau defined in the proof of \cite[Lemma~3.5]{PfluegerkGonal} has no vertical steps.  An example of such a tableau appears in Figure~\ref{Fig:NoVerticalSteps}.  Although the symbol 22 appears directly above the symbol 23, it is larger than $g-k$, so this tableau satisfies the definition.

\begin{figure}[H]
\begin{ytableau}
*(gray)1 & *(gray)2 & *(gray)3 & 7 & 8 & 9 & 13 \\
*(gray)4 & *(gray)5 & *(gray)6 & 10 & 11 & 12 & 16 \\
*(gray)7 & *(gray)8 & *(gray)9 & 13 & 14 & 15 & 19 \\
*(gray)10 & *(gray)11 & *(gray)12 & *(gray)16 & *(gray)17 & *(gray)18 & *(gray)22 \\
*(gray)13 & *(gray)14 & *(gray)15 & *(gray)19 & *(gray)20 & *(gray)21 & *(gray)23 \\
\end{ytableau}
\caption{A scrollar tableau with no vertical steps.}
\label{Fig:NoVerticalSteps}
\end{figure}

The following proposition is one of the main technical results of \cite{JensenRanganthan}.  In this proposition and throughout Section~\ref{Sec:Connections}, we let $K$ be an algebraically closed, non-archimedean valued field of equicharacteristic zero.

\begin{proposition} \cite[Proposition~9.2]{JensenRanganthan}
\label{Prop:JRLifting}
Let $t$ be a scrollar tableau of type $\alpha$ with no vertical steps, and let $D \in \mathbb{T}(t)$ be a sufficiently general divisor class.  Then there exists a curve $C$ of genus $g$ and gonality $k$ over $K$ with skeleton $\Gamma$, and a divisor class $\cD \in W^r_d(C)$ specializing to $D$.
\end{proposition}

\section{Connections Between Combinatorics and Algebraic Geometry}
\label{Sec:Connections}

In this section, we demonstrate the connection between our combinatorial and geometric results.  Specifically, we show that Theorem~\ref{thm:MaxScrollar} implies Theorem~\ref{thm:Strata}.  To begin, we establish the connection between scrollar tableaux of type $\alpha$ and the splitting types $\bmu_{\alpha}$.

\begin{proposition}
\label{Prop:ScrollSpecialize}
Let $C$ be a curve of genus $g$ and gonality $k$ over $K$ with skeleton $\Gamma$.  Let $t$ be a scrollar tableau of type $\alpha$, let $D \in \mathbb{T}(t)$ be a sufficiently general divisor class, and let $\cD \in W^r_d (C)$ be a divisor that specializes to $D$.  Then $\cD \in W^{\bmu_{\alpha}}(C)$.
\end{proposition}

\begin{proof}
Let $\bmu$ denote the splitting type of $\pi_* \cO (\cD)$.  By Lemma~\ref{Lem:JRBounds}, we have 
\begin{align*}
   \rk(D-qg^1_k) &= \beta-1, \\
   \rk(D-(q+1)g^1_k) &= -1.
\end{align*}

By Baker's Specialization Lemma \cite{Baker}, it follows that
\begin{align}
\label{eq:q} h^0 (\cD-qg^1_k) &\leq \beta, \\
\label{eq:q+1} h^0 (\cD-(q+1)g^1_k) &= 0.
\end{align}

Recall that, if $t^T$ denotes the transpose of $t$, then the Serre dual $K_{\Gamma} - D $ is contained in $\mathbb{T}(t^T)$.  Note that $t^T$ is also a scrollar tableau.  By Lemma~\ref{Lem:JRBounds}, therefore, since $K_{\Gamma} - D$ is sufficiently general, we see that
\begin{align*}
  \rk(K_{\Gamma}-D-q'g^1_k) &= \beta'-1, \\
  \rk(K_{\Gamma}-D-(q'+1)g^1_k) &= -1.
\end{align*}

By Baker's Specialization Lemma, it follows that
\begin{align}
\label{eq:q'} h^0 (K_C -\cD-q'g^1_k) &\leq \beta', \\
\label{eq:q'+1} h^0 (K_C -\cD-(q'+1)g^1_k) &= 0.
\end{align}

By (\ref{eq:ranksequence}), (\ref{eq:q+1}) implies that $\mu_k \leq q$ and (\ref{eq:q}) implies that $\mu_{k-\beta} \leq q-1$.  It follows that
\[
\mu_{k-\alpha+1} + \cdots + \mu_{k-\alpha+ \ell} \leq \mu_{\alpha,k-\alpha+1} + \cdots + \mu_{\alpha,k-\alpha+ \ell} \hspace{.2 in} \mbox{ for all } \ell \leq \alpha.
\]

Similarly, (\ref{eq:q'+1}) implies that $\mu_1 \geq -q'-2$, and (\ref{eq:q'}) implies that $\mu_{\beta'+1} \geq -q'-1$.  It follows that
\[
\mu_1 + \cdots + \mu_{\ell} \geq \mu_{\alpha,1} + \cdots + \mu_{\alpha,\ell} \hspace{.2 in} \mbox{ for all } \ell \leq k-\alpha.
\]
Putting these together, we see that $\bmu \geq \bmu_{\alpha}$.  By Proposition~\ref{Prop:Incomparable}, however, $\bmu_{\alpha}$ is maximal, hence $\bmu = \bmu_{\alpha}$.
\end{proof}

\begin{corollary}
\label{Cor:Lifting}
Let $t$ be a scrollar tableau of type $\alpha$ with no vertical steps, and let $D \in \mathbb{T}(t)$ be a sufficiently general divisor class.  Then there exists a curve $C$ of genus $g$ and gonality $k$ over $K$ with skeleton $\Gamma$, and a divisor class $\cD \in W^{\mu_{\alpha}}(C)$ specializing to $D$.
\end{corollary}

\begin{proof}
By Proposition~\ref{Prop:JRLifting}, there exists a curve $C$ of genus $g$ and gonality $k$ over $K$ with skeleton $\Gamma$, and a divisor class $\cD \in W^r_d(C)$ specializing to $D$.  By Proposition~\ref{Prop:ScrollSpecialize}, the divisor class $\cD$ is in $W^{\mu_{\alpha}}(C)$.
\end{proof}

We now show that Theorem~\ref{thm:MaxScrollar} implies Theorem~\ref{thm:Strata}.  We do this in two steps.  First, we obtain an upper bound on a particular component of $W^r_d (C)$.

\begin{proposition}
\label{prop:DimensionUpperBound}
Let $C$ and $\cD$ be as in Corollary~\ref{Cor:Lifting}, and let $\mathcal{Y}$ be any irreducible component of $W^r_d(C)$ containing $\cD$.  Then
\[
\dim \mathcal{Y} \leq g - \vert \bmu_{\alpha} \vert .
\]
\end{proposition}

\begin{proof}
By \cite[Theorem~6.9]{Gubler07},
\[
\dim \mathcal{Y} = \dim \Trop \mathcal{Y} .
\]
By Baker's Specialization Lemma, we see that $\Trop \mathcal{Y} \subseteq W^r_d (\Gamma)$.  It follows that $\dim \mathcal{Y}$ cannot exceed the local dimension of $W^r_d (\Gamma)$ in a neighborhood of $D$.  By Theorem~\ref{thm:MaxScrollar}, $\mathbb{T}(t)$ is maximal with respect to containment in $W^r_d (\Gamma)$, and since $D \in \mathbb{T}(t)$ is sufficiently general, the local dimension of $W^r_d(\Gamma)$ in a neighborhood of $D$ is equal to that of $\mathbb{T}(t)$.  Finally, by Proposition~\ref{Prop:ScrollarDim}, we have
\[
\dim \mathcal{Y} \leq \dim \mathbb{T}(t) = g - \vert \bmu_{\alpha} \vert .
\]
\end{proof}

\begin{proof}[Proof that Theorem~\ref{thm:MaxScrollar} implies Theorem~\ref{thm:Strata}]
The case $k=2$ is classical, so we assume that $k \geq 3$.  Let $\alpha \leq \min \{ r+1, k-1 \}$ be a positive integer satisfying either $\alpha \geq k-(g-d+r)$ or $\alpha = r+1$.  If $\alpha \geq k-(g-d+r)$, then applying Serre duality exchanges $\alpha$ with $k-\alpha$, so we may assume that $\alpha > 1$.

Since $\vert \bmu \vert \leq g$ and $\alpha > 1$, there exists a scrollar tableau $t$ of type $\alpha$ with no vertical steps.  Let $D \in \mathbb{T}(t)$ be a sufficiently general divisor class.  By Corollary~\ref{Cor:Lifting}, there exists a curve $C$ of genus $g$ and gonality $k$ over $K$ with skeleton $\Gamma$, and a divisor class $\cD \in W^{\mu_{\alpha}}(C)$ specializing to $D$.  If $\mathcal{Y}$ is an irreducible component of $W^{\mu_{\alpha}}(C)$ containing $\cD$, then by Proposition~\ref{prop:DimensionUpperBound}, we have
\[
\dim \mathcal{Y} \leq g - \vert \bmu_{\alpha} \vert .
\]
It therefore suffices to prove the reverse inequality.

The rest of the proof is identical to that of \cite[Theorem~9.3]{JensenRanganthan}, which we reproduce here for the sake of completeness.  Let $\mathcal{M}_g^k$ be the moduli space of curves of genus $g$ that admit a degree $k$ map to $\PP^1$, let $\mathcal{C}_k$ be the universal curve, and let $\mathcal{W}^{\bmu_{\alpha}}$ be the universal splitting-type locus over $\mathcal{M}_g^k$.  Let $\widetilde{\mathcal{W}}^{\bmu_{\alpha}}$ be the locus in the symmetric $d$th fiber power of $\mathcal{C}_k$ parameterizing divisors $\cD$ such that $\pi_* \cO(\cD)$ has splitting type $\bmu_{\alpha}$.

We work in the Berkovich analytic domain of $k$-gonal curves whose skeleton is a $k$-gonal chain of loops.  By Corollary~\ref{Cor:Lifting}, the tropicalization of $\widetilde{\mathcal{W}}^{\bmu_{\alpha},\an}$ has dimension at least
\[
3g-5+2k-\vert \bmu_{\alpha} \vert +r .
\]
If $\pi_* \cO(\cD) \cong \cO (\bmu_{\alpha})$, then $\cD$ has rank exactly $r$.  It follows that $\mathcal{W}^{\bmu_{\alpha}}$ has dimension at least
\[
3g-5+2k-\vert \bmu_{\alpha} \vert .
\]
By Corollary~\ref{Cor:Lifting}, there is an irreducible component of $\mathcal{W}^{\bmu_{\alpha}}$ whose tropicalization contains pairs of the form $(\Gamma, \cD)$ where $\Gamma$ is a $k$-gonal chain of loops and $\cD \in \mathbb{T}(t)$ is sufficiently general.  The image of this component in $\mathcal{M}_g^{k,\trop}$ has dimension $2g-5+2k$.  It follows that this component dominates $\mathcal{M}_g^k$, and the fibers have dimension at least $g-\vert \bmu_{\alpha} \vert$.

Combining the two bounds, we see that there exists an irreducible component $\mathcal{Y}$ of $W^{\bmu_{\alpha}}(C)$, containing $\cD$, of dimension $g- \vert \bmu_{\alpha} \vert$.  If $\mathcal{Z}$ is a component of $W^r_d (C)$ containing $\mathcal{Y}$, then by Proposition~\ref{prop:DimensionUpperBound}, we see that
\[
\dim \mathcal{Z} = \dim \mathcal{Y}.
\]
It follows that $\mathcal{Z}$ is the closure of $\mathcal{Y}$.
\end{proof}

\section{Maximality of Scrollar Tableaux}
\label{Sec:Max}

Having established that Theorem~\ref{thm:Strata} follows from our combinatorial results, it remains to prove the combinatorial results.  The goal of this section is to prove the following.

\begin{theorem}
\label{thm:Maximal} 
Let $t$ be a scrollar tableau of type $\alpha$ on $[a]\times [b]$.  Then $\mathbb{T}(t)$ is maximal with respect to containment.
\end{theorem}

Before proving Theorem~\ref{thm:Maximal}, we first make two simple observations.  These will be useful because, if $\mathbb{T}(t) \subseteq \mathbb{T}(t')$, then by Lemma~\ref{lem:Containment}, for every box $(n,m)$ in $[a]\times[b]$, there exists a box $(x,y)$ such that $t'(n,m) = t(x,y)$.  Our argument will break into cases, depending on the location of $(x,y)$ relative to that of $(n,m)$.

\begin{lemma}
\label{lem:Regions} 
Let $\alpha$ be a positive integer and $(n,m)$ any box in $[a]\times[b]$.  For any box $(x,y)$ in $[a]\times[b]$, there exists an integer $\ell$ such that one of the following holds:
\begin{enumerate}
    \item $x \leq n-\ell\alpha$ and $y \leq m+\ell(k-\alpha)$, 
    \item $x \geq n-\ell\alpha$ and $y \geq m+\ell(k-\alpha)$, or
    \item $n-(\ell+1)\alpha < x < n-\ell\alpha$ and $m+\ell(k-\alpha) < y < m+(\ell+1)(k-\alpha)$.
\end{enumerate}
\end{lemma}

\begin{proof}
By the division algorithm, there exists an integer $\ell$ such that 
\[
n-(\ell+1)\alpha < x \leq n-\ell\alpha .
\]
If $y \leq m+\ell(k-\alpha)$, then case (1) holds. If $y \geq m+(\ell+1)(k-\alpha)$, or if $x = n-\ell\alpha$ and $y \geq m+\ell(k-\alpha)$, then case (2) holds.  Otherwise, $x \neq n-\ell\alpha$, and case (3) holds.
\end{proof}

Lemma~\ref{lem:Regions} is illustrated in Figure~\ref{Fig:Regions}.  Boxes of the form $(n-\ell\alpha, m+\ell(k-\alpha))$ are labeled with stars, and the three cases of Lemma~\ref{lem:Regions} are depicted in gray.  Note that every box is contained in one of the three gray regions.

\begin{figure}[H]
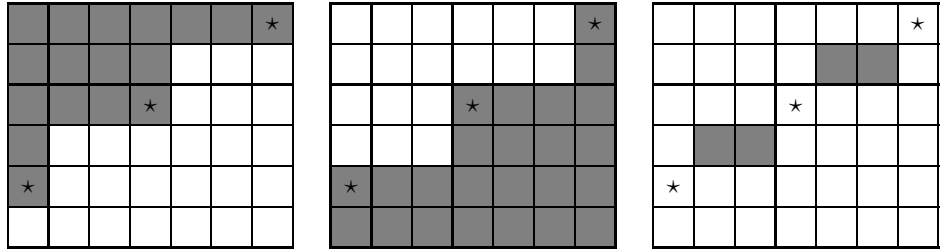

\begin{ytableau}
*(gray){ } & *(gray){ } & *(gray){ } & *(gray){ } & *(gray){ } & *(gray){ } & *(gray)\star \\
*(gray){ } & *(gray){ } & *(gray){ } & *(gray) & { } & { } & { } \\
*(gray){ } & *(gray){ } & *(gray){ } & *(gray)\star & { } & { } & { } \\
*(gray) & { } & { } & { } & { } & { } & { } \\
*(gray)\star & { } & { } & { } & { } & { } & { } \\
{ } & { } & { } & { } & { } & { } & { } \\
\end{ytableau}
\hspace{.1in}
\begin{ytableau}
{ } & { } & { } & { } & { } & { } & *(gray)\star \\
{ } & { } & { } & { } & { } & { } & *(gray){ } \\
{ } & { } & { } & *(gray)\star & *(gray){ } & *(gray){ } & *(gray){ } \\
{ } & { } & { } & *(gray){ } & *(gray){ } & *(gray){ } & *(gray){ } \\
*(gray)\star & *(gray){ } & *(gray){ } & *(gray){ } & *(gray){ } & *(gray){ } & *(gray){ } \\
*(gray){ } & *(gray){ } & *(gray){ } & *(gray){ } & *(gray){ } & *(gray){ } & *(gray){ } \\
\end{ytableau}
\hspace{.1in}
\begin{ytableau}
{ } & { } & { } & { } & { } & { } & \star \\
{ } & { } & { } & { } & *(gray){ } & *(gray){ } & { } \\
{ } & { } & { } & \star & { } & { } & { } \\
{ } & *(gray){ } & *(gray){ } & { } & { } & { } & { } \\
\star & { } & { } & { } & { } & { } & { } \\
{ } & { } & { } & { } & { } & { } & { } \\
\end{ytableau}

\caption{The three regions described in Lemma~\ref{lem:Regions}.}
\label{Fig:Regions}
\end{figure}

\begin{remark}
If $\alpha$ is equal to either $a$ or $k-b$, then the integer $\ell$ in Lemma~\ref{lem:Regions} can be taken to be one of $-1$, $0$, or $1$, as illustrated in Figure~\ref{Fig:EdgeCase}.  If $\ell = \pm 1$, then the box $(n-\ell\alpha,m+\ell(k-\alpha))$ is not contained in $[a]\times[b]$.
\end{remark}

\begin{figure}[H]
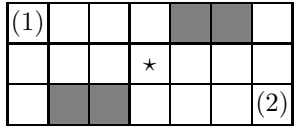

\begin{ytableau}
(1) & { } & { } & { } & *(gray){ } & *(gray){ } & { } \\
{ } & { } & { } & \star & { } & { } & { } \\
{ } & *(gray){ } & *(gray){ } & { } & { } & { } & (2) \\
\end{ytableau}

\caption{When $\alpha = k-b$, the integer $\ell$ can be taken to be one of $-1$, $0$, or $1$.}
\label{Fig:EdgeCase}
\end{figure}

The following simple lemma is key to our argument.

\begin{lemma}
\label{lem:TooClose} 
Let $\alpha$ be a positive integer, let $(n,m)$ be any box in $[a]\times[b]$, and $(x,y)$ a box satisfying condition (3) of Lemma~\ref{lem:Regions}.  Then 
\[
x-y\not\equiv n-m\pmod k.
\]
\end{lemma}

\begin{proof}
Since 
\[
n-(\ell+1)a < x < n-\ell a
\]
and
\[
m+\ell(k-a) < y < m+(\ell+1)(k-a),
\]
we have 
\[
(n-m)-(\ell+1)k < x-y < (n-m)-\ell k.
\]
Hence $x-y\not\equiv n-m\pmod k$. 
\end{proof}

We are now prepared to prove the main result of this section, the maximality of scrollar tableaux.

\begin{proof}[Proof of Theorem~\ref{thm:Maximal}]
Let $t'$ be a $k$-uniform displacement tableau such that $\mathbb{T}(t)\subseteq\mathbb{T}(t')$.  We will show that $t=t'$.  We first demonstrate, by induction, that $t'(n,m)\geq t(n,m)$ for all $(n,m) \in [a]\times[b]$.  The base case $t'(1,1)\geq t(1,1)$ holds because, by Lemma~\ref{lem:Containment}, $t'(1,1)$ must be a symbol in $t$, and $t(1,1)$ is the smallest symbol in $t$.

For our inductive hypothesis, suppose that $t'(x,y) \geq t(x,y)$ for all $(x,y)$ such that $x\leq n$ and $y\leq m$, not both equal.  We will show that $t'(n,m)\geq t(n,m).$  By Lemma~\ref{lem:Containment}, there exists $(x,y) \in [a]\times[b]$ such that $t'(n,m)=t(x,y)$.  By Lemma~\ref{lem:Regions}, there exists an integer $\ell$ such that one of the following holds:
\begin{enumerate}
    \item $x \leq n-\ell\alpha$ and $y \leq m+\ell(k-\alpha)$, 
    \item $x \geq n-\ell\alpha$ and $y\geq m+\ell(k-\alpha)$, or
    \item $n-(\ell+1)\alpha < x < n-\ell\alpha$ and $m+\ell(k-\alpha) < y < m+(\ell+1)(k-\alpha)$.
\end{enumerate}
If $(x,y)$ satisfies (3), then by Lemma~\ref{lem:TooClose}, $x-y \not\equiv n-m \pmod k$, a contradiction to Lemma~\ref{lem:Containment}.  Hence $(x,y)$ must satisfy either (1) or (2).

There are now two cases to consider -- the case where the box $(n-\ell\alpha,m+\ell(k-\alpha))$ is contained in $[a]\times[b]$, and the case where it is not.  We first consider the case where $(n-\ell\alpha,m+\ell(k-\alpha))$ is contained in $[a]\times[b]$.  Notice that, if $\alpha$ is equal to $a$ or $k-b$, then in this case we must have $\ell = 0$.  If $(x,y)$ satisfies (2), then 
\[
t(x,y) \geq t(n-\ell\alpha,m+\ell(k-\alpha)) = t(n,m),
\]
hence $t'(n,m) \geq t(n,m)$, as desired.  If $(x,y)$ satisfies (1) and $(x,y) \neq (n-\ell\alpha,m+\ell(k-\alpha))$, we have either 
\begin{align*}
t'(n,m) &= t(x,y) \leq t(n-\ell\alpha-1,m+\ell(k-\alpha)), \mbox{ or} \\
t'(n,m) &= t(x,y) \leq t(n-\ell\alpha,m+\ell(k-\alpha)-1) .
\end{align*}
First, assume that $n,m > 1$.  Since $t$ is scrollar, we have 
\begin{align*}
t(n-\ell\alpha-1,m+\ell(k-\alpha)) &= t(n-1,m) \mbox{ and } \\
t(n-\ell\alpha,m+\ell(k-\alpha)-1) &= t(n,m-1) .
\end{align*}
By our inductive hypothesis, however, we have $t(n-1,m) \leq t'(n-1,m)$ and $t(n,m-1) \leq t'(n,m-1)$.  This guarantees that either $t'(n,m) \leq t'(n-1,m)$ or $t'(n,m) \leq t'(n,m-1)$, a contradiction.  It follows that
\[
t'(n,m) = t(x,y) = t(n-\ell\alpha,m+\ell(k-\alpha)) = t(n,m).
\]
Now, suppose that $m=1$ and $n>1$.  The case where $n=1$ will follow from a similar argument.  Without loss of generality, let $\ell$ be the smallest integer such that $(x,y)$ is above and to the right of $(n-\ell\alpha,m+\ell(k-\alpha))$.  If $x < n-\ell\alpha$, then the conclusion follows from the argument above.  On the other hand, if $x = n-\ell\alpha$, then since 
\[
m + (\ell-1)(k-\alpha) < y < m+\ell(k-\alpha),
\]
we see that $x-y \neq n-m \pmod{k}$, a contradiction to Lemma~\ref{lem:Containment}.

We now turn to the case where $(n-\ell\alpha,m+\ell(k-\alpha))$ is not contained in $[a]\times[b]$.  First, suppose that $(x,y)$ satisfies (2).  In this case, either $n-\ell\alpha \leq 0$ or $m+\ell(k-\alpha) \leq 0$, but not both.  We will assume that $n-\ell\alpha \leq 0$; the other case follows by a similar argument.  If $\ell$ is any integer satisfying $n-\ell\alpha \leq 0$, then $(x,y)$ is below and to the right of $(1,m+\ell(k-\alpha))$.  We may therefore assume without loss of generality that $\ell$ is the minimal integer such that $n-\ell\alpha \leq 0$.  If $\alpha$ is equal to $a$, then $\ell = 1$.  Because $t$ is scrollar, we observe that
\begin{align*}
t'(n,m) = t(x,y) &\geq t(1,m+\ell(k-\alpha)) > t(\alpha,m+(\ell-1)(k-\alpha)) \\
&\geq t(n-(\ell-1)\alpha,m+(\ell-1)(k-\alpha)) = t(n,m) .
\end{align*}

Now, suppose that $(x,y)$ satisfies (1).  In this case, either $b < m+\ell(k-\alpha)$ or $a < n-\ell\alpha$, but not both.  We will assume that $b < m+\ell(k-\alpha)$.  The other case follows by a similar argument.  Without loss of generality, assume that $\ell$ is the minimal integer such that $b < m + \ell(k-\alpha)$.  As above, if $\alpha = k-b$, then $\ell = 1$.  If $y \leq m+(\ell-1)(k-\alpha)$, then by replacing $\ell$ with $\ell-1$, we may reduce to the case where $(n-\ell\alpha,m+\ell(k-\alpha))$ is in $[a]\times[b]$.  We may therefore assume that
\[
m+(\ell-1)(k-\alpha) < y \leq b < m + \ell(k-\alpha).
\]
This situation is illustrated in Figure~\ref{Fig:Outside}.  The boxes $(n-\ell\alpha,m+\ell(k-\alpha))$ and $(n-(\ell-1)\alpha,m+(\ell-1)(k-\alpha))$ are labeled with stars, the box $(n-\ell\alpha,b)$ is labeled with a diamond, and the box $(x,y)$ is located somewhere in the shaded region.

\begin{figure}[h!]
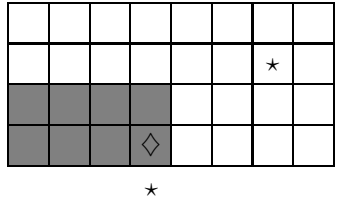


\begin{ytableau}
{ } & { } & { } & { } & { } & { } & { } & { } \\
{ } & { } & { } & { } & { } & { } & \star & { } \\
*(gray){ } & *(gray){ } & *(gray){ } & *(gray){ } & { } & { } & { } & { } \\
*(gray){ } & *(gray) & *(gray){ } & *(gray)\diamondsuit & { } & { } & { } & { } \\
\end{ytableau}\\
\vspace{0.075in} \hspace{-0.25in} $\star$

\caption{An illustration of the case where $(n-\ell\alpha,m+\ell(k-\alpha))$ is not contained in $[a]\times[b]$.}
\label{Fig:Outside}
\end{figure}

If $x = n-\ell\alpha$, then since
\[
m+(\ell-1)(k-\alpha) < y < m + \ell(k-\alpha),
\]
we see that $x-y \not\equiv n-m \pmod k$, a contradiction to Lemma~\ref{lem:Containment}.  We may therefore assume that $x < n-\ell\alpha$.  Because $t$ is scrollar, we have
\begin{align*}
t'(n,m) = t(x,y) &\leq t(n-\ell\alpha-1,b) < t(n-(\ell-1)\alpha-1,b+1-(k-\alpha)) \\
&\leq t(n-(\ell-1)\alpha-1,m+(\ell-1)(k-\alpha)) = t(n-1,m).
\end{align*}
By induction, however, we have $t(n-1,m) \leq t'(n-1,m)$, hence $t'(n,m) \leq t'(n-1,m)$, a contradiction.

Thus, in every case we see that $t'(n,m) \geq t(n,m)$.  We now show that $t'(n,m)\leq t(n,m)$ for all $(n,m) \in [a]\times[b]$.  Combining the two inequalities, we see that $t'=t$.  Given a tableau $t$, define the ``rotated'' tableau $t_R$ as follows:
\[
t_R (x,y) = g+1-t(a+1-x,b+1-y) .
\]
(See Figure~\ref{Fig:Rotate} for an example.)
Returning to our tableaux $t$ and $t'$, we see that by definition, both $t_R$ and $t'_R$ are $k$-uniform displacement tableaux, the tableau $t_R$ is scrollar, and $\mathbb{T}(t_R) \subseteq \mathbb{T}(t'_R)$.  By the argument above, we see that $t'_R(n,m)\geq t_R(n,m)$ for all $(n,m) \in [a]\times[b]$, hence $t'(n,m)\leq t(n,m)$ for all $(n,m) \in [a]\times[b]$, and the conclusion follows.
\end{proof}

\begin{figure}[H]
\begin{ytableau}
*(gray)1 & *(gray)2 & *(gray)4 & 5 & 10 & 11 & 12 \\
*(gray)3 & *(gray)7 & *(gray)8 & 9 & 13 & 16 & 18 \\
*(gray)5 & *(gray)10 & *(gray)11 & 12 & 15 & 17 & 20 \\
*(gray)9 & *(gray)13 & *(gray)16 & *(gray)18 & *(gray)19 & *(gray)22 & *(gray)23 \\
*(gray)12 & *(gray)15 & *(gray)17 & *(gray)20 & *(gray)21 & *(gray)24 & *(gray)26 \\
\end{ytableau}
\hspace{.2in}
\begin{ytableau}
*(gray)26 & *(gray)24 & *(gray)21 & 20 & 17 & 15 & 12 \\
*(gray)23 & *(gray)22 & *(gray)19 & 18 & 16 & 13 & 9 \\
*(gray)20 & *(gray)17 & *(gray)15 & 12 & 11 & 10 & 5 \\
*(gray)18 & *(gray)16 & *(gray)13 & *(gray)9 & *(gray)8 & *(gray)7 & *(gray)3 \\
*(gray)12 & *(gray)11 & *(gray)10 & *(gray)5 & *(gray)4 & *(gray)2 & *(gray)1 \\
\end{ytableau}
\hspace{.2in}
\begin{ytableau}
*(gray)1 & *(gray)3 & *(gray)6 & 7 & 10 & 12 & 15 \\
*(gray)4 & *(gray)5 & *(gray)8 & 9 & 11 & 14 & 18 \\
*(gray)7 & *(gray)10 & *(gray)12 & 15 & 16 & 17 & 22 \\
*(gray)9 & *(gray)11 & *(gray)14 & *(gray)18 & *(gray)19 & *(gray)20 & *(gray)24 \\
*(gray)15 & *(gray)16 & *(gray)17 & *(gray)22 & *(gray)23 & *(gray)25 & *(gray)26 \\
\end{ytableau}
\caption{To obtain the ``rotation'' of the tableau on the left, first rotate 180 degrees, and then subtract each entry from $g+1$.}
\label{Fig:Rotate}
\end{figure}

\section{Non-Existence of Other Maximal Tableaux}
\label{Sec:NonMax}

In this section, we prove the following.

\begin{theorem}
\label{thm:NotMax}
Let $t$ be a $k$-uniform displacement tableau on $[a]\times[b]$.  Then there exists a scrollar tableau $t'$ on $[a]\times[b]$ such that $\mathbb{T}(t) \subseteq \mathbb{T}(t')$.
\end{theorem}

\noindent Together with Theorem~\ref{thm:Maximal}, this establishes Theorem~\ref{thm:MaxScrollar}.

To prove Theorem~\ref{thm:NotMax}, we will describe an algorithm that, starting with $t$, produces a scrollar tableau $t'$ by replacing certain symbols in $t$ with other symbols in $t$.  We first introduce a statistic on the boxes in a $k$-uniform displacement tableau.

\begin{definition}
\label{defn:Statistic}
Let $t$ be a $k$-uniform displacement tableau on $[a]\times[b]$.  Given a box $(x,y)$ such that $x+y \geq k$, we define a statistic $S_t(x,y)$ as follows.  Consider the symbols appearing above $(x,y)$ in column $x$ and to the left of $(x,y)$ in row $y$.  Among these symbols, the $k-1$ largest ones form a hook of width $\alpha$ and height $k-\alpha$.  We define $S_t(x,y)$ to be $\alpha$.
\end{definition}

\begin{remark}
Note that if $x+y < k$, then $S_t(x,y)$ is undefined.  In this case the box $(x,y)$ is left empty.  Additionally, the statistic $\alpha$ cannot appear in any box $(x,y)$ with $x < \alpha$ or $y < k-\alpha$.  In particular, for any $k$-uniform displacement tableau $t$, we have $S_t(\alpha,k-\alpha) = \alpha$.

Note also that $S_t$ is well-defined.  To see this let $i$ be the smallest positive integer such that $t(x-i,y) = t(x,y-j)$ for some positive integer $j$. By the definition of a $k$-uniform displacement tableau, $i+j$ must be a multiple of $k$.  It follows that the hook from $(x-i,y)$ to $(x,y-j)$ contains at least $k-1$ distinct symbols, all greater than $t(x-i,y)$.
\end{remark}

\begin{example}\label{example:Working}
Figure~\ref{Fig:Statistic} depicts an example of a $5$-uniform displacement tableau $t$ on $[4]\times[4]$.  The first figure is $t$, the second is $S_t$, and the last two depict example hooks of width 2 and 3, respectively.

\begin{figure}[H]
\begin{ytableau}
1 & 2 & 3 & 9\\
4 & 6 & 7 & 10\\
5 & 8 & 11 & 13\\
12 & 14 & 15 & 16\\
\end{ytableau}
\begin{ytableau}
\none & & &  & 4\\
\none &  &  & 3 & 3\\
\none & & 2 & 3 & 2 \\
\none & 1 & 2 & 3 & 3\\
\end{ytableau}\hskip4ex
\begin{ytableau}
1 & 2 & 3 & 9\\
4 & *(gray)6 & 7 & 10\\
5 & *(gray)8 & 11 & 13\\
*(gray)12 & *(gray)14 & 15 & 16\\
\end{ytableau}\hskip4ex
\begin{ytableau}
1 & 2 & 3 & 9\\
 4 & 6 & 7 & 10\\
5 & 8 & 11 & *(gray)13\\
12 & *(gray)14 & *(gray)15 & *(gray)16\\
\end{ytableau}
\caption{A 5-uniform displacement tableau, its associated statistics, and some example hooks.}
\label{Fig:Statistic}
\end{figure}
\end{example}

Before proceeding further, we will first need the following property of the statistic $S_t$.

\begin{lemma}
\label{lem:BadConfigurations}
Let $t$ be a $k$-uniform displacement tableau.  We have the following inequalities on statistics:
\begin{align*}
S_t (x+1,y) &\leq S_t (x,y)+1 \\
S_t (x,y-1) &\leq S_t (x,y)+1 \\
S_t (x+1,y-1) &\leq S_t (x,y)+1.
\end{align*}
\end{lemma}

\begin{proof}
Let $H$ be the hook containing the $k-1$ largest symbols appearing above $(x,y)$ in column $x$ and to the left of $(x,y)$ in row $y$.  By the definition of $S_t$, $H$ contains the boxes 
\[
(x,y+1-k+S_t(x,y)) \mbox{ and } (x+1-S_t(x,y),y),
\]
but not the boxes
\[
(x,y-k+S_t(x,y)) \mbox{ or } (x-S_t(x,y),y).
\]
It follows that
\[
t(x-S_t(x,y),y) < t(x,y+1-k+S_t(x,y)) \mbox{ and }
\]
\[
t(x+1-S_t(x,y),y) > t(x,y-k+S_t(x,y)).
\]
If $S_t(x+1,y) > S_t(x,y)+1$, then
\begin{align*}
t(x-S_t(x,y),y) \geq t(x+2-S_t(x+1,y),y) \\
> t(x+1,y-k+S_t(x+1,y)) > t(x,y+1-k+S_t(x,y)),
\end{align*}
a contradiction.

Similarly, if $S_t(x,y-1) > S_t(x,y)+1$, then
\begin{align*}
t(x+1-S_t(x,y),y) < t(x-S_t(x,y-1),y-1) \\
< t(x,y-k+S_t(x,y-1)) \leq t(x,y-k+S_t(x,y)),
\end{align*}
a contradiction.

Finally, if $S_t(x+1,y-1) > S_t(x,y)+1$, then
\begin{align*}
t(x-S_t(x,y),y) > t(x+2-S_t(x+1,y-1),y-1) \\
> t(x+1,y-1-k+S_t(x+1,y-1)) > t(x,y-k+S_t(x,y)),
\end{align*}
another contradiction.
\end{proof}

\begin{definition}\label{defn:Paths}
Let $t$ be a $k$-uniform displacement tableau on $[a]\times[b]$, and suppose that $a+b \geq k$.  An \emph{admissible path} $P$ of type $\alpha$ in $t$ is a sequence of boxes
\[
P = (x_0,y_0) , (x_1,y_1) , \ldots , (x_{a+b-k},y_{a+b-k})
\]
satisfying the following conditions:
\begin{enumerate}
    \item $(x_0,y_0) = (\alpha,k-\alpha)$ and $(x_{a+b-k},y_{a+b-k}) = (a,b)$.
    \item For all $i$, $(x_i,y_i)$ is equal to either $(x_{i-1}+1,y_{i-1})$ or $(x_{i-1},y_{i-1}+1)$. 
    \item If $(x_i,y_i) = (x_{i-1}+1,y_{i-1})$, then $S_t(x_i,y_i) \leq \alpha$.
    \item If $(x_i,y_i) = (x_{i-1},y_{i-1})+1$, then $S_t(x_i,y_i) \geq \alpha$.
\end{enumerate}
\end{definition}

In other words, an admissible path is a sequence of pairwise adjacent boxes starting at $(\alpha,k-\alpha)$ and ending in the bottom right corner of the tableau.  Every time the path moves right, the statistic in the new box must be at most $\alpha$, and every time the path moves down, the statistic in the new box must be at least $\alpha$.

\begin{example}
\label{example:Paths}

Figure~\ref{Fig:Paths} depicts the statistics $S_t$ for the tableau $t$ from Example~\ref{example:Working}, together with two admissible paths of type 3 shaded.  Note that the first path is admissible because the box labeled 2 is to the right of the previous box in the path.

\begin{figure}[H]
\begin{ytableau}
\none & & &  & 4\\
\none &  &  & *(gray)3 & 3\\
\none & & 2 & *(gray)3 & *(gray)2 \\
\none & 1 & 2 & 3 & *(gray) 3\\
\end{ytableau}\hskip2ex\begin{ytableau}
\none & & &  & 4\\
\none &  &  & *(gray)3 & 3\\
\none & & 2 & *(gray)3 & 2 \\
\none & 1 & 2 & *(gray)3 & *(gray) 3\\
\end{ytableau}
\caption{Two admissible paths of type 3}
\label{Fig:Paths}
\end{figure}
\end{example}

Note that an admissible path of type $a$ is completely vertical, and an admissible path of type $k-b$ is completely horizontal.  If there is an admissible path of type $a$ in a tableau $t$, then $S_t(a,y) = a$ for all $y \geq k-a$.  It follows that $t(1,y+k-a) > t(a,y)$ for all $y > k-a$, so $t$ is scrollar of type $a$.  Similarly, if there is an admissible path of type $k-b$ in a tableau $t$, then $t$ is scrollar of type $k-b$.

The first main goal of this section is to prove the existence of admissible paths.  That is, given a $k$-uniform tableau $t$ on $[a]\times[b]$, we show that there exists an integer $\alpha$ and a admissible path $P$ of type $\alpha$.  Our argument will require the following lemma.

\begin{lemma}
\label{lem:NonCrossing}
Let $t$ be a $k$-uniform displacement tableau.  If $P_1$ and $P_2$ are two admissible paths in $t$ of types $\alpha_1$ and $\alpha_2$, respectively, then $\alpha_1 = \alpha_2$.
\end{lemma}

\begin{proof}
First, note that the last box in any admissible path is $(a,b)$, so any two admissible paths intersect.  Let $(x,y)$ be the box in the intersection that minimizes $x+y$.  Without loss of generality, assume that $\alpha_1 > \alpha_2$.  Note that $P_1$ starts at $(\alpha_1,k-\alpha_1)$, which is above and to the right of $(\alpha_2,k-\alpha_2)$.  Because $(x,y)$ is the first box at which the two paths cross, we see that $P_1$ must contain the box $(x,y-1)$ and $P_2$ must contain the box $(x-1,y)$.  By the definition of admissible paths, we have
\[
\alpha_1 \leq S_t(x,y) \leq \alpha_2,
\]
contradicting our assumption that $\alpha_1 > \alpha_2$.
\end{proof}

We now prove that admissible paths exist.

\begin{proposition}
\label{prop:AdmissibleExistence}
Let $t$ be a $k$-uniform displacement tableau on $[a]\times[b]$, and suppose that $a+b \geq k$.  Then there exists an admissible path in $t$.
\end{proposition}

\begin{proof}
We proceed by induction on $a+b$.  In the base case $b=k-a$, the admissible path consists of the single box $(a,b)$.

If $a+b > k$, then by induction the tableau $t_1$ obtained by deleting the last row of $t$ contains an admissible path $P_1$ of type $\alpha_1$.  Similarly, the tableau $t_2$ obtained by deleting the last column of $t$ contains an admissible path $P_2$ of type $\alpha_2$.  We will show that either the path $P'_1$ obtained by appending $(a,b)$ to $P_1$ or the path $P'_2$ obtained by appending $(a,b)$ to $P_2$ is admissible.  Note that $P'_1$ is admissible if and only if $S_t(a,b) \geq \alpha_1$ and $P'_2$ is admissible if and only if $S_t(a,b) \leq \alpha_2$.

If $S_t(a,b) < S_t(a,b-1)$, then by Lemma~\ref{lem:BadConfigurations}, we have 
\[
S_t(a,b-1) = S_t(a,b)+1 \mbox{ and }
\]
\[
S_t(a-1,b) \geq S_t(a,b-1)-1 = S_t(a,b).
\]
It follows that either $S_t(a,b) \geq S_t(a,b-1)$ or $S_t(a,b) \leq S_t(a-1,b)$.  We assume that $S_t(a,b) \geq S_t(a,b-1)$; the case where $S_t(a,b) \leq S_t(a-1,b)$ follows by a similar argument.  If $S_t(a,b) \geq \alpha_1$, then $P'_1$ is an admissible path of type $\alpha_1$, and we are done.  If $P_1$ contains the box $(a,b-2)$, then $S_t(a,b) \geq S_t(a,b-1) \geq \alpha_1$, by the definition of an admissible path.  We may therefore assume that $P_1$ contains the box $(a-1,b-1)$, and $S_t(a,b) < \alpha_1$.

Now consider the path $P_2$.  If the paths $P_1$ and $P_2$ intersect, let $(x,y)$ be a box in the intersection, and let $t_3$ be the tableau obtained by restricting $t$ to $[x]\times[y]$.  The restrictions of $P_1$ and $P_2$ to $t_3$ are both admissible, and it follows from Lemma~\ref{lem:NonCrossing} that $\alpha_1 = \alpha_2$.  Since $S_t(a,b) < \alpha_1$, we see that $P'_2$ is an admissible path.

If $P_1$ and $P_2$ do not intersect, then $P_1$ lies entirely above and to the right of $P_2$, so $\alpha_1 > \alpha_2$.  Let $(x,b-1)$ be the leftmost box of $P_1$ in row $b-1$.  Because the two paths do not intersect, the boxes $(x-1,b)$ and $(x,b)$ must be contained in $P_2$.  By the definition of admissible paths, we have $\alpha_1 \leq S_t (x,b-1)$ and $\alpha_2 \geq S_t (x,b)$.  By Lemma~\ref{lem:BadConfigurations}, however, we have
\[
\alpha_1 \leq S_t(x,b-1) \leq S_t(x,b)+1 \leq \alpha_2 + 1.
\]
It follows that $\alpha_1 = \alpha_2 + 1$.  Since $S_t(a,b) < \alpha_1$, we see that $S_t(a,b) \leq \alpha_1 + 1 = \alpha_2$, hence $P'_2$ is an admissible path.
\end{proof}

Now that we know admissible paths exist, we can use them to construct a scrollar tableau from an arbitrary tableau.

\begin{example}
Before proving Theorem~\ref{thm:NotMax}, we first illustrate the idea with an example.  Figure~\ref{Fig:Algorithm} depicts the example of a 5-uniform displacement tableau $t$ and an admissible path of type $\alpha = 3$ from Example~\ref{example:Paths}.  The proof of Theorem~\ref{thm:NotMax} provides us with an iterative procedure for constructing a scrollar tableau $t'$ of type 3 such that $\mathbb{T}(t) \subseteq \mathbb{T}(t')$.  This procedure begins with the subtableau on $[\alpha] \times [k-\alpha] = [3]\times[2]$.  It then follows the admissible path, extending the tableau one row or one column at a time.  Every time we extend the tableau by a column, we replace each symbol in the new column with the symbol appearing $\alpha$ boxes to the left and $k-\alpha$ boxes below in the previous tableau.  Similarly, every time we extend the tableau by a row, we replace each symbol in the new row with the symbol appearing $\alpha$ boxes to the right and $k-\alpha$ boxes above in the previous tableau.  The definition of admissible paths guarantees that this construction yields a tableau.

\begin{figure}[H]
\begin{ytableau}
1 & 2 & 3 & 9\\
4 & 6 & *(gray)7 & 10\\
5 & 8 & *(gray)11 & *(gray)13\\
12 & 14 & 15 & *(gray)16\\
\end{ytableau}\hskip2ex
\begin{ytableau}
1 & 2 & 3\\
4 & 6 & 7\\
\end{ytableau}\hskip2ex
\begin{ytableau}
1 & 2 & 3\\
4 & 6 & 7\\
5 & 8 & 11
\end{ytableau}\hskip2ex
\begin{ytableau}
1 & 2 & 3 & 5\\
4 & 6 & 7 & 10\\
5 & 8 & 11 & 13\\
\end{ytableau}\hskip2ex
\begin{ytableau}
1 & 2 & 3 & 5\\
4 & 6 & 7 & 10\\
5 & 8 & 11 & 13\\
10 & 14 & 15 & 16\\
\end{ytableau}
\caption{Construction of a scrollar tableau from a given $k$-uniform displacement tableau and admissible path.}
\label{Fig:Algorithm}
\end{figure}

\end{example}

\begin{proof}[Proof of Theorem~\ref{thm:NotMax}]
First, note that if $a+b \leq k$, then $t$ is scrollar of type $a$ for trivial reasons.  We therefore assume that $a+b > k$.  By Proposition~\ref{prop:AdmissibleExistence}, there exists an admissible path $P$ in $t$ of type $\alpha$.  We will prove, by induction on $a+b$, that there exists a scrollar tableau $t'$ of type $\alpha$ on $[a]\times[b]$ such that $\mathbb{T}(t) \subseteq \mathbb{T}(t')$.  In addition, we will see that $t'(a-i,b) = t(a-i,b)$ for all $i < \alpha$ and $t'(a,b-j) = t(a,b-j)$ for all $j < k-\alpha$.  We assume that $P$ contains the box $(a-1,b)$; the case where $P$ contains the box $(a,b-1)$ follows by a similar argument.  By the definition of admissible paths, this implies that $S_t(a,b) \leq \alpha$.

Let $t_1$ be the tableau obtained by deleting the last column from $t$.  The restriction of $P$ to $t_1$ is an admissible path of type $\alpha$ in $t_1$.  By induction, therefore, there exists a scrollar tableau $t'_1$ on $[a-1]\times[b]$ such that $\mathbb{T}(t_1) \subseteq \mathbb{T}(t'_1)$.  Moreover, we have $t'_1(a-1-i,b) = t_1(a-1-i,b)$ for all $i<\alpha$, and $t'_1(a-1,b-j) = t_1(a-1,b-j)$ for all $j < k-\alpha$.  By Lemma~\ref{lem:Containment}, every symbol in $t'_1$ is a symbol in $t_1$, and if $t_1(x,y) = t'_1(x',y')$, then $x-y = x'-y'\pmod k$.

We now define a tableau $t'$ on $[a]\times[b]$.
\begin{displaymath}
t'(x,y) = \left\{ \begin{array}{ll}
t'_1(x,y) & \textrm{if $x < a$} \\
t'_1(x-\alpha,y+k-\alpha) & \textrm{if $x=a$ and $y \leq b-k+\alpha$} \\
t(x,y) & \textrm{if $x=a$ and $y > b-k+\alpha$.}
\end{array}\right.
\end{displaymath}
We first show that $t'$ is a tableau.  Let $(x,y) \in [a]\times[b]$.  If $x<a$, then since $t'_1$ is a tableau, we see that $t'(x,y) > t'(x-1,y)$ and $t'(x,y) > t'(x,y-1)$.  If $y \leq b-k+\alpha$, then because $t'_1$ is a tableau, we have $t'(a,y) > t'(a,y-1)$, and because $t'_1$ is scrollar of type $\alpha$, we have
\[
t'(a-1,y) < t'(a-\alpha,y+k-\alpha) = t'(a,y).
\]
If $y > b-k+\alpha$, then since $t$ is a tableau and $t'_1(a-1,y) = t(a-1,y)$, we have $t'(a-1,y) < t'(a,y)$.  If $y > b+1-k+\alpha$, then since $t$ is a tableau, we have $t'(a,y-1) < t'(a,y)$.  Finally, since $S_t(a,b) \leq \alpha$, we have
\[
t'(a,b-k+\alpha) = t(a-\alpha,b) < t(a,b+1-k+\alpha) = t'(a,b+1-k+\alpha).
\]

To see that $t'$ is scrollar, we show that if $b > k-\alpha$, then $t'(x,y) = t'(x+\alpha,y-k+\alpha)$ for all pairs $(x,y)$.  This is clear if $x+\alpha < a$, because $t'_1$ is scrollar of type $\alpha$.  On the other hand, if $x+\alpha = a$, then this holds by construction.  If $\alpha = k-b$, then $t'(x,1) > t'(x+b-k,b)$ for all $x<a$ because $t'_1$ is scrollar, and $t'(a,1) > t'(a+b-k,b)$ because $S_t(a,b) \leq \alpha = k-b$.

Finally, we show that $\mathbb{T}(t) \subseteq \mathbb{T}(t')$.  Note that the symbol $t'(x,y)$ is also a symbol in $t_1$ if and only if $x<a$ or $y \leq b-k+\alpha$.  By construction, every symbol in $t_1$ is also a symbol in $t$, and if
\[
t'(x,y) = t_1(x',y') = t (x'',y''),
\]
then
\[
x-y = x'-y' = x''-y''\pmod k.
\]
On the other hand, if $y > b-k+\alpha$, then the symbol $t'(a,y) = t(a,y)$ appears only in one box, and there is nothing to prove.  By Lemma~\ref{lem:Containment}, it follows that $\mathbb{T}(t) \subseteq \mathbb{T}(t')$.
\end{proof}

\bibliography{references}{}
\bibliographystyle{alpha}
\end{document}